\documentclass[oneside]{amsart}
\usepackage{amssymb}
\usepackage{amsxtra}
\usepackage{amsmath}
\usepackage{amsthm}
\usepackage[all]{xy}

\usepackage[pagebackref]{hyperref}

\usepackage[english]{babel}

\DeclareMathOperator{\PGL}{PGL}

\DeclareMathOperator{\Sim}{Sim}

\DeclareMathOperator{\HH}{H}

\DeclareMathOperator{\XX}{X}

\DeclareMathOperator{\im}{Im\,}

\DeclareMathOperator{\hgt}{ht}

\DeclareMathOperator{\PGO}{PGO}

\newcommand{\pp}{\mathfrak p}

\newtheorem{lem}{Lemma}
\newtheorem*{thm*}{Theorem}
\newtheorem{thm}{Theorem}

\newtheorem*{cor*}{Corollary}

\title{Rationally isotropic exceptional projective homogeneous varieties are locally isotropic}
\thanks{The authors are partially supported by RFBR~12-01-92695. The second
author is partially supported by RFBR~12-01-31100.}
\author{I. Panin}
\address{PDMI RAS, St. Petersburg, Russia;\\
Chebyshev Laboratory at St. Petersburg State University, Russia}
\email{paniniv@gmail.com}
\author{V. Petrov}
\address{PDMI RAS, St. Petersburg, Russia;\\
Chebyshev Laboratory at St. Petersburg State University, Russia}
\email{victorapetrov@googlemail.com}
\begin{document}

\begin{abstract}
Assume that $R$ is a local regular ring containing an infinite perfect
field, or that $R$ is the local ring of a point on a smooth scheme over an infinite field.
Let $K$ be the field of fractions of $R$ and $char(K) \neq 2$.
Let $X$ be an exceptional projective homogeneous scheme over $R$.
We prove that in most cases the condition $X(K)\neq\emptyset$ implies $X(R)\neq\emptyset$.
\end{abstract}

\maketitle

\section{Introduction}

We prove the following theorem.

\begin{thm}\label{thm:main}
Let $R$ be
local regular ring containing an infinite perfect field, or
the local ring of a point on a smooth scheme over an infinite field,
$K$ be the fraction field of $R$ and $char(K) \neq 2$. Let $G$ be a split simple group of
exceptional type (that is, $E_6$, $E_7$, $E_8$, $F_4$, or $G_2$), $P$ be a
parabolic subgroup of $G$, $\xi$ be a class from $\HH^1(R,G)$, and
$X={}_\xi(G/P)$ be the corresponding homogeneous space.
Assume that $P\ne P_7,\,P_8,\,P_{7,8}$ in case $G=E_8$, $P\ne P_7$ in case
$G=E_7$, and $P\ne P_1$ in case $G=E_7^{ad}$. Then the condition
$X(K)\ne\emptyset$ implies $X(R)\ne\emptyset$.
\end{thm}

\section{Purity of some $\HH^1$ functors}

Let $A$ be a commutative noetherian domain of finite Krull dimension with a fraction field $F$.
We say that a functor $\mathcal F$ from the category of commutative $A$-algebras
to the category of sets \emph{satisfies purity} for $A$ if we have
$$
\im[{\mathcal F}(A)\to{\mathcal F}(F)]=\bigcap_{\hgt{\pp}=1}\im[{\mathcal F}(A_{\pp})\to{\mathcal F}(F)].
$$
If $\mathcal H$ is an \'{e}tale group sheaf we write
$\HH^i(-,\mathcal H)$
for
$\HH^i_{\text{\'et}}(-, \mathcal H)$
below through the text.
The following theorem is proven in the characteristic zero case
\cite[Theorem~4.0.3]{Pa-pur2}.
We extend it here to an arbitrary characteristic case.
\begin{thm}\label{thm:quot}
Let $R$ be the regular local ring from Theorem
\ref{thm:main}
and $k \subset R$
be the subfield of $R$ mentioned in that Theorem. Let
\begin{equation}\label{eq:seq}\tag{*}
1\to Z\to G\to G'\to 1
\end{equation}
be an exact sequence of algebraic $k$-groups, where $G$ and $G'$ are reductive and $Z$ is a closed central
subgroup scheme in $G$. If the functor $\HH^1(-,G')$ satisfies purity for $R$ then the functor
$\HH^1(-,G)$ satisfies purity for $R$ as well.
\end{thm}

\begin{lem}
\label{Purity}
Consider the category of $R$-algebras. The functor
$$
R' \mapsto {\mathcal F}(R')=\HH^1_{fppf}(R',Z)/\im(\delta_{R'}),
$$
where $\delta$ is the connecting homomorphism associated to sequence \eqref{eq:seq},
satisfies purity for $R$.
\end{lem}
\begin{proof}
Similar to the proof of \cite[Theorem~12.0.36]{Pa-newpur}
\end{proof}

\begin{lem}
\label{Lemma2}
The map
$$
\HH^2_{fppf}(R,Z)\to\HH^2_{fppf}(K,Z)
$$
is injective.
\end{lem}
\begin{proof}
Similar to the proof of \cite[Theorem~13.0.38]{Pa-newpur}
\end{proof}

\begin{proof}[Proof of Theorem~\ref{thm:quot}]
Reproduce the diagram chase from the proof of
\cite[Theorem~4.0.3]{Pa-pur2}.
For that consider the commutative diagram
$$
\xymatrix {
\{1\} \ar[r] & {{\mathcal F}(K)} \ar[r]^{\delta_K}  & {\HH^1_{}(K,G)} \ar[r]^{\pi_K} &
{\HH^1(K,G')} \ar[r]^{\Delta_K} & \HH^2_{fppf}(K,Z) \\
\{1\} \ar[r] & {{\mathcal F}(R)} \ar[r]^{\delta} \ar[u]^{\alpha} & {\HH^1_{}(R,G)} \ar[r]^{\pi} \ar[u]_{\beta} &
{\HH^1_{}(R,G')} \ar[u]_{\gamma} \ar[r]^{\Delta} &  \HH^2_{fppf}(R,Z) \ar[u]_{\alpha_1}\\
}
$$
Let
$\xi \in \HH^1_{}(K,G)$ be an
$R$-unramified class and let
$\bar \xi=\pi_K (\xi)$.
Clearly,
$\bar \xi \in \HH^1_{}(K,G')$
is $R$-unramified.
Thus there exists an element
$\bar \xi' \in \HH^1_{}(R,G')$
such that
$\bar \xi'_K=\bar \xi$.
The map $\alpha_1$
is injective by
Lemma \ref{Lemma2}.
One has
$\Delta (\bar \xi')=0$,
since
$\Delta_K (\bar \xi)=0$.
Thus there exists
$\xi' \in \HH^1_{}(R,G)$
such that
$\pi (\xi')=\bar \xi'$.

The group
${\mathcal F}(K)$
acts on the set
$\HH^1_{}(K,G)$,
since $Z$ is a central subgroup of the group $G$.
If
$a \in {\mathcal F}(K)$
and
$\zeta \in \HH^1_{}(K,G)$,
then we will write
$a\cdot \zeta$
for the resulting element in
$\HH^1_{}(K,G)$.

Further, for any two elements
$\zeta_1, \zeta_2 \in \HH^1_{}(K,G)$,
having the same images in
$\HH^1_{}(K,G)$
there exists a unique element
$a \in {\mathcal F}(K)$
such that
$a\cdot \zeta_1 = \zeta_2$.
These remarks hold for the cohomology of the ring $R$, and for the cohomology of the rings $R_{\mathfrak p}$,
where $\mathfrak p$ runs over all height 1 prime ideals of $R$.
Since the images of $\xi'_K$ and $\xi$ coincide in
$\HH^1_{}(K,G^{\prime})$,
there exists a unique element
$a \in {\mathcal F}(K)$
such that
$a \cdot \xi'_K= \xi$
in
$\HH^1_{}(K,G)$.

\begin{lem}
\label{Claim}
The above constructed element
$a \in {\mathcal F}(K)$
is an
$R$-unramified.
\end{lem}
Assuming Lemma \ref{Claim} complete the proof of the Theorem.
By Lemma \ref{Purity} the functor
${\mathcal F}$
satisfies the purity for regular local rings containing the field $k$.
Thus there exists an element
$a' \in {\mathcal F}(R)$
with
$a'_K= a$.
It's clear that
$\xi''= a' \cdot \xi' \in \HH^1_{}(R,G)$
satisfies the equality
$\xi''_K = \xi$.
It remains to prove Lemma \ref{Claim}.

For that consider a height 1 prime ideal
$\mathfrak p$ in $R$.
Since $\xi$ is an $R$-unramified there exists its lift up to an element
$\tilde \xi$
in
$\HH^1_{}(R_{\mathfrak p},G)$.
Set
$\xi^{\prime}_{\mathfrak p}$
to be the image of the
$\xi^{\prime}$
in
$\HH^1_{}(R_{\mathfrak p},G)$.
The classes
$\pi_{\mathfrak p}(\tilde \xi), \pi_{\mathfrak p}(\xi^{\prime}_{\mathfrak p}) \in \HH^1_{}(R_{\mathfrak p},G')$
being regarded in
$\HH^1_{}(K,G')$
coincide.

The map
$$
\HH^1_{}(R_{\mathfrak p},G') \to \HH^1_{}(K,G')
$$
is injective by Lemma
\ref{Injectivity},
formulated and proven below.
Thus
$\pi_{\mathfrak p}(\tilde \xi) = \pi_{\mathfrak p}(\xi'_{\mathfrak p})$.
Therefore there exists a unique class
$a_{\mathfrak p} \in {\mathcal F}(R_{\mathfrak p})$
such that
$a_{\mathfrak p} \cdot \xi'_{\mathfrak p}=\tilde \xi \in \HH^1_{}(R_{\mathfrak p},G)$.
So,
$a_{\mathfrak p,K} \cdot \xi'_K= \xi \in \HH^1_{}(K,G)$
and
$a_{\mathfrak p,K} \cdot \xi'_K = \xi = a \cdot \xi'_K$.
Thus
$a= a_{\mathfrak p,K}$.
Lemma
\ref{Claim}
is proven.

\end{proof}

To finish the proof of the theorem it remains to prove the following
\begin{lem}
\label{Injectivity}
Let $H$ be a reductive group scheme over a discrete valuation ring $A$.
Assume that for each $A$-algebra $\Omega$ with an algebraically closed field $\Omega$ the algebraic group
$H_{\Omega}$ is connected.
Let $K$ be the fraction field of $A$.
Then the map
$$
\HH^1_{}(R,H) \to \HH^1_{}(K,H)
$$
is injective.
\end{lem}

\begin{proof}
Let
$\xi_0,\xi_1 \in \HH^1_{}(A,H)$.
Let
${\mathcal H}_0$ be a principal homogeneous $H$-bundle representing the class $\xi_0$.
Let
$H_0$ be the inner form of the group scheme $H$, corresponding to
${\mathcal H}_0$.
Let $X=Spec(A)$. For each $X$-scheme $S$ there is a well-known bijection
$\phi_S\colon \HH^1_{}(S,H) \to \HH^1_{}(S,H_0)$
of non-pointed sets.
That bijection takes the
principal homogeneous $H$-bundle
${\mathcal H}_0 \times_X S$
to the trivial principal homogeneous $H_0$-bundle
$H_0 \times_X S$.
That bijection respects to morphisms of
$X$-schemes.

Assume that
$\xi_{0,K}=\xi_{1,K}$. Then one has
$* = \phi_K (\xi_{0,K})= \phi_K (\xi_{1,K}) \in \HH^1_{}(K,H_0)$.
The kernel of the map
$\HH^1_{}(A,H_0) \to \HH^1_{}(K,H_0)$
is trivial by Nisnevich theorem
\cite{Ni}. Thus
$\phi_A (\xi_1)= * =\phi_A(\xi_0) \in \HH^1_{}(A,H_0)$.
Whence
$\xi_1= \xi_0 \in \HH^1_{}(A,H)$.
The Lemma is proven and the Theorem is proven as well.

\end{proof}

\begin{thm}\label{thm:pgl}
The functor $\HH^1(-,\PGL_n)$ satisfies purity.
\end{thm}
\begin{proof}
Let $\xi \in \HH^1(-,\PGL_n)$ be an $R$-unramified element. Let
$\delta: \HH^1(-,\PGL_n) \to \HH^2(-, G_m)$
be the boundary map corresponding to the short exact sequence of \'{e}tale group sheaves
$$1 \to \mathbb G_m \to \textrm{GL}_n \to \PGL_n \to 1.$$
Let $D_{\xi}$ be a central simple $K$-algebra of degree $n$ corresponding $\xi$.
If
$D_{\xi} \cong M_l(D')$
for a skew-field $D'$, then there exists
$\xi' \in \HH^1_{}(K,\textrm{PGL}_{n'})$
such that
$D'=D_{\xi'}$. Then
$\delta (\xi')=[D']=[D]=\delta (\xi)$.
Replacing $\xi$ by $\xi'$, we may assume that
$D:=D_{\xi}$ is a central skew-field over $K$ of degree $n$ and the class $[D]$ is $R$-unramified.

Clearly, the class $\delta (\xi)$ is $R$-unramified. Thus
there exists an Azumaya $R$-algebra $A$ and an integer $d$ such that $A_K = M_d(D)$.

There exists a projective left $A$-module $P$ of finite rank such that
each projective left $A$-module $Q$ of finite rank is isomorphic to the left
$A$-module $P^m$ for an appropriative integer $m$
(see \cite[Cor.2]{DeM}). In particular,
two projective left $A$-modules of finite rank are isomorphic if they have the same rank as $R$-modules.
One has an isomorphism
$A \cong P^s$
of left $A$-modules for an integer $s$.
Thus one has $R$-algebra isomorphisms
$A \cong \textrm{End}_A(P^s) \cong \textrm{M}_s(\textrm{End}_A(P))$.
Set
$B=\textrm{End}_A(P)$.
Observe, that
$B_K=\textrm{End}_{A_K}(P_K)$,
since $P$ is a finitely generated projective left $A$-module.

The class $[P_K]$ is a free generator of the group
$K_0(A_K)=K_0(M_d(D)) \cong \mathbb Z$,
since $[P]$
is a free generator of the group $K_0(A)$ and
$K_0(A)=K_0(A_K)$.
The
$P_K$
is a simple $A_K$-module, since
$[P_K]$
is a free generator of
$K_0(A_K)$.
Thus
$\textrm{End}_{A_K}(P_K)=B_K$
is a skew-field.

We claim that the $K$-algebras
$B_K$ and $D$
are isomorphic.
In fact,
$A_K=M_r(B_K)$ for an integer $r$, since $P_K$ is a simple $A_K$-module.
From the other side $A_K=M_d(D)$. As $D$, so $B_K$ are skew-fields. Thus
$r=d$ and $D$ is isomorphic to $B_K$ as $K$-algebras.

We claim further that $B$ is an Azumaya $R$-algebra.
That claim is local with respect to the \'{e}tale topology on
$\textrm{Spec}(R)$.
Thus it suffices to check the claim assuming that
$\textrm{Spec}(R)$
is sticktly henselian local ring.
In that case
$A=M_l(R)$ and $P=(R^{l})^m$
as an $M_l(R)$-module.
Thus
$B=\textrm{End}_A(P)=M_m(R)$,
which proves the claim.

Since $B_K$ is isomorphic to $D$, one has $m=n$.
So, $B$ is an Azumaya $R$-algebra, and the $K$-algebra $B_K$ is isomorphic to $D$.
Let
$\zeta \in \HH^1_{}(R,\textrm{PGL}_n)$
be class corresponding to $B$. Then
$\zeta_K=\xi$, since
$\delta (\zeta)_K = [B_K]=[D]=\delta (\xi) \in \HH^2_{}(K,\mathbb G_m)$.

\end{proof}

\begin{thm}\label{thm:sim}
The functor $\HH^1(-,\Sim_n)$ satisfies purity.
\end{thm}
\begin{proof}
Let $\varphi$ be a quadratic form over $K$ whose similarity class represents
$\xi\in\HH^1(K,\Sim_n)$.
Diagonalizing $\varphi$ we may assume that
$\varphi= \sum^n_{i=1}f_i\cdot t^2_i$
for certain non-zero elements
$f_1, f_2, \dots , f_n \in K$.
For each $i$ write $f_i$ in the form
$f_i=\frac {g_i}{h_i}$
with
$g_i, h_i \in R$
and $h_i \neq 0$.

There are only finitely many
height one prime ideals
$\frak q$ in $R$
such that there exists
$0 \leq i \leq n$
with
$f_i$ not in $R_{\frak q}$.
Let
$\frak q_1, \frak q_2, \dots , \frak q_s$
be all height one prime ideals in $R$
with that property and let
$\frak q_i \neq \frak q_j$
for $i \neq j$.

For all other hight one prime ideals
$\frak p$ in $R$
each $f_i$ belongs to
the group of units
$R^{\times}_{\frak p}$
of the ring
$R_{\frak p}$.

If $\frak p$ is a hight one prime ideal of $R$ which is not from the list
$\frak q_1, \frak q_2, \dots , \frak q_s$,
then
$\varphi= \sum^n_{i=1}f_i\cdot t^2_i$
may be regarded as a quadratic space over
$R_{\frak p}$. We will write
$_{\frak p}\varphi$
for that quadratic space over $R_{\frak p}$.
Clearly, one has
$(_{\frak p}\varphi) \otimes_{R_{\frak p}} K = \varphi$
as quadratic spaces over $K$.

For each
$j \in \{1,2, \ldots, s \}$
choose and fix a quadratic space
$_j \varphi$
over
$R_{\frak q_j}$
and a non-zero element
$\lambda_j \in K$
such that
the quadratic spaces
$(_j \varphi) \otimes_{R_{\frak q_j}} K$
and
$\lambda_j \cdot \varphi$
are isomorphic over $K$.
The ring $R$ is factorial since it is regular and local.
Thus for each
$j \in \{1,2, \ldots, s \}$
we may choose an element
$\pi_j \in R$
such that firstly
$\pi_j$ generates the only maximal ideal in
$R_{\frak q_j}$
and secondly
$\pi_j$ is an invertible element in
$R_{\frak n}$
for each hight one prime ideal
$\frak n$
different from the ideal
$\frak q_j$.

Let
$v_j\colon K^{\times} \to \mathbb Z$
be the discrete valuation of $K$
corresponding to the prime ideal
$\frak q_j$.
Set
$\lambda = \prod^s_{i=1} \pi^{v_j(\lambda_j)}_j$
and
$$\varphi_{new} = \lambda \cdot \varphi.$$

{\it Claim.}
The quadratic space $\varphi_{new}$ is $R$-unramified.
In fact,
if a hight one prime ideal $\frak p$ is different from each of $\frak q_j$'s,
then
$v_{\frak p}(\lambda)=0$.
Thus,
$\lambda \in R^{\times}_{\frak p}$.
In that case
$\lambda \cdot (_\frak p \varphi)$
is a quadratic space over
$R_{\frak p}$
and moreover one have isomorphisms of quadratic spaces
$(\lambda \cdot (_\frak p \varphi)) \otimes_{R_{\frak p}} K = \lambda \cdot \varphi = \varphi_{new}$.
If we take one of $\frak q_j$'s,
then
$\frac {\lambda}{\lambda_j} \in R^{\times}_{\frak q_j}$.
Thus,
$\frac {\lambda}{\lambda_j} \cdot (_j \varphi)$
is a quadratic space over $R_{\frak q_j}$.
Moreover,
one has
$$\frac {\lambda}{\lambda_j} \cdot (_j \varphi) \otimes_{R_{\frak q}} K=
\frac {\lambda}{\lambda_j} \cdot \lambda_j \cdot \varphi = \varphi_{new}.$$
The Claim is proven.

By
\cite[Corollary 3.1]{PP-sus}
there exists a quadratic space
$\tilde \varphi$
over $R$
such that the quadratic spaces
$\tilde \varphi \otimes_R K$
and
$\varphi_{new}$
are isomorphic over $K$.
This shows that the similarity classes of the quadratic spaces
$\tilde \varphi \otimes_R K$
and
$\varphi$
coincide.
The theorem is proven.
\end{proof}

\begin{thm}\label{thm:sim+}
The functor $\HH^1(-,\Sim_n^+)$ satisfies purity.
\end{thm}
\begin{proof}
Consider an element $\xi\in\HH^1(K,\Sim_n^+)$ such that for any $\pp$ of
height $1$ $\xi$ comes from $\xi_{\pp}\in\HH^1(R_{\pp},\Sim_n^+)$.
Then the image of $\xi$ in $\HH^1(K,\Sim_n)$ by Theorem~\ref{thm:sim}
comes from some $\zeta\in\HH^1(R,\Sim_n)$. We have a short exact sequence
$$
1\to\Sim_n^+\to\Sim_n\to\mu_2\to 1,
$$
and
$R^{\times}/(R^{\times})^2$
injects into
$K^{\times}/(K^{\times})^2$.
Thus the element
$\zeta$ comes actually from some $\zeta'\in\HH^1(R,\Sim_n^+)$. It remains
to show that the map
$$
\HH^1(K,\Sim_n^+)\to\HH^1(K,\Sim_n)
$$
is injective, or, by twisting, that the map
$$
\HH^1(K,\Sim^+(q))\to\HH^1(K,\Sim(q))
$$
has the trivial kernel. The latter follows from the fact that the map
$$
\Sim(q)(K)\to\mu_2(K)
$$
is surjective (indeed, any reflection goes to $-1\in\mu_2(K)$).
\end{proof}

\section{Proof of the Main Theorem}

Let $\xi$ be a class from $\HH^1(R,G)$, and
$X={}_\xi(G/P)$ be the corresponding homogeneous space.
Denote by $L$ a Levi subgroup of $P$.

\begin{lem}\label{lem:morph}
Let $L$ modulo its center be isomorphic to $\PGO_{2m}^+$ (resp., $\PGO_{2m+1}^+$
or $\PGO_{2m}^+\times\PGL_2$). Denote by $\Psi$ the closed subset in $\XX^*(T)$
of type $D_m$ (resp. $B_m$ or $D_m+A_1$) corresponding to $L$, $T$ stands for a
maximal split torus in $L$. Assume that there is an element $\lambda\in\XX^*(T)$
such that $\Psi$ and $\lambda$ generate a closed subset of type $D_{m+1}$
(resp. $B_{m+1}$ or $D_{m+1}+A_1$), and $\Psi$ forms the standard subsystem of
type $D_m$ (resp. $B_m$ or $D_m+A_1$) therein. Then there is a surjective map
$L\to\Sim_{2m}^+$ (resp., $L\to\Sim_{2m+1}^+$ or $L\to\Sim_{2m}^+\times\PGL_2$)
whose kernel is a central closed subgroup scheme in $L$. In particular, the functor
$\HH^1(-,L)$ satisfies purity.
\end{lem}
\begin{proof}
It is easy to check that $\Sim_{2m}^+$ (resp. $\Sim_{2m+1}^+$)
is a Levi subgroup in the split adjoint group of type $D_{m+1}$
(resp. $B_{m+1}$). Now the first claim follows from
\cite[Exp.~XXIII, Thm.~4.1]{SGA}, and the rest follows from
Theorem~\ref{thm:sim+} and Theorem~\ref{thm:pgl}.
\end{proof}

\begin{lem}\label{lem:Levi}
For any $R$-algebra $S$ the map
$$
\HH^1(S,L)\to\HH^1(S,G)
$$
is injective. Moreover, $X(S)\ne\emptyset$ if and only if
$\xi_S$ comes from $\HH^1(S,L)$.
\end{lem}
\begin{proof}
See \cite[Exp.~XXVI, Cor.~5.10]{SGA}.
\end{proof}

\begin{lem}\label{lem:main}
Assume that the functor $\HH^1(-,L)$ satisfies purity.
Then $X(K)\ne\emptyset$ implies $X(R)\ne\emptyset$.
\end{lem}
\begin{proof}
By Lemma~\ref{lem:Levi} $\xi_K$ comes from some $\zeta\in\HH^1(K,L)$,
which is uniquely determined. Since $X$ is smooth projective, for any
prime ideal $\pp$ of height $1$ we have $X(R_\pp)\ne\emptyset$. By
Lemma~\ref{lem:Levi} $\xi_{R_\pp}$ comes from some
$\zeta_\pp\in\HH^1(R_\pp,L)$. Now $(\zeta_\pp)_K=\zeta$, and so
by the purity assumption there is $\zeta'\in\HH^1(R,L)$ such that
$\zeta'_K=\zeta$.

Set $\xi'$ to be the image of $\zeta'$ in $\HH^1(R,G)$. We claim that
$\xi'=\xi$.
To prove this recall that
by the construction $\xi'_K=\xi_K$. Further,
there are natural in $R$-algebras bijections
$\alpha_S: \HH^1(S,G) \to \HH^1(R,{}_{\xi'}G)$,
which takes the $\xi'$ to the distinguished element
$*_R \in \HH^1(R,{}_{\xi'}G)$.
The $R$-group scheme ${}_{\xi'}G$ is isotropic and one has equalities
$$(\alpha_R(\xi))_K=\alpha_K(\xi_K)=\alpha_K(\xi'_K)=*'_K \in \HH^1(K,{}_{\xi'}G).$$
Thus by
\cite[Theorem~1.0.1]{Pa-newpur}
one has
$$\alpha_R(\xi)=*'=\alpha_R(\xi') \in \HH^1(R,{}_{\xi'}G).$$
The $\alpha_R$ is a bijection, whence
$$\xi=\xi' \in \HH^1(R,G).$$
Lemma is proved.

\end{proof}

\begin{lem}\label{lem:taut}
Let $Q\le P$ be another parabolic subgroup, $Y={}_\xi(G/Q)$.
Assume that $X(K)\ne\emptyset$ implies $Y(K)\ne\emptyset$, and
$Y(K)\ne\emptyset$ implies $Y(R)\ne\emptyset$. Then $X(K)\ne\emptyset$
implies $X(R)\ne\emptyset$.
\end{lem}
\begin{proof}
Indeed, there is a map $Y\to X$, so $Y(R)\ne\emptyset$ implies $X(R)\ne\emptyset$.
\end{proof}

\begin{proof}[Proof of Theorem~\ref{thm:main}]
By Lemma~\ref{lem:taut} we may assume that $P$ corresponds to an item from
the list of Tits \cite[Table~II]{Tits66}. We show case by case that $\HH^1(-,L)$
satisfies purity, hence we are in the situation of Lemma~\ref{lem:main}.

If $P=B$ is the Borel subgroup, obviously $\HH^1(-,L)=1$. In the case
of index $E_{7,4}^9$ (resp. ${}^1E_{6,2}^{16}$) $L$ modulo its center is
isomorphic to $\PGL_2\times\PGL_2\times\PGL_2$ (resp. $\PGL_3\times\PGL_3$), and
we may apply Theorem~\ref{thm:quot} and Theorem~\ref{thm:pgl}. In the all other
cases we provide an element $\lambda\in\XX^*(T)$ such that the assumption of
Lemma~\ref{lem:morph} holds ($\tilde\alpha$ stands for the maximal root,
enumeration follows \cite{Bu}). The
indices $E_{7,1}^{78}$, $E_{8,1}^{133}$ and $E_{8,2}^{78}$ are not in the list below
since in those cases the $L$ does not belong to one of the type $D_m$, $B_m$, $D_m \times A_1$.
The index $E_{7,1}^{66}$ is not in the list below since in that case
we need a weight $\lambda$ which is not in the
root lattice. So, the indices
$E_{7,1}^{78}$, $E_{8,1}^{133}$, $E_{8,2}^{78}$ and $E_{7,1}^{66}$
are the exceptions in the statement of the Theorem.

\medskip

\begin{tabular}{c|c}
Index&$\lambda$\\\hline
${}^1E_{6,2}^{28}$&$\alpha_1$\\
$E_{7,1}^{66}$&$-\omega_7$\\
$E_{7,1}^{48}$&$-\tilde\alpha$\\
$E_{7,2}^{31}$&$\alpha_1$\\
$E_{7,3}^{28}$&$\alpha_1$\\
$E_{8,1}^{91}$&$-\tilde\alpha$\\
$E_{8,2}^{66}$&$\alpha_8$\\
$E_{8,4}^{28}$&$\alpha_1$\\
$F_{4,1}^{21}$&$-\tilde\alpha$
\end{tabular}

\end{proof}

\bigskip
The authors heartily thank Anastasia Stavrova for discussion on the earlier version of this work.

\end{document}